\documentclass[12pt]{amsart}
 
\usepackage{amsfonts,latexsym,rawfonts,amsmath,amssymb,amsthm,mathrsfs}
 
\RequirePackage{color}
\theoremstyle{plain}

\newtheorem{theorem}{Theorem}[section]
\newtheorem{lemma}[theorem]{Lemma}

\theoremstyle{definition}

\theoremstyle{claim}

\theoremstyle{remark}
\newtheorem{remark}[theorem]{Remark}

\numberwithin{equation}{section}

\begin{document}

\title{A note on the singular set of the no-sign obstacle problem}

\author 
{Shibing Chen, Yibin Feng and Yuanyuan Li}

\address{Shibing Chen, School of Mathematical Sciences,
University of Science and Technology of China,
Hefei, 230026, P.R. China.}
\email{chenshib@ustc.edu.cn}

\address{Yibin Feng, School of Mathematical Sciences,
	University of Science and Technology of China,
	Hefei, 230026, P.R. China.}
\email{fybt1894@ustc.edu.cn}

\address{Yuanyuan Li, School of Mathematical Sciences,
	University of Science and Technology of China,
	Hefei, 230026, P.R. China.}
\email{lyysa20001025@mail.ustc.edu.cn}


\thanks{The first author would like to thank Professor  Shahgholian for his valuable suggestions.}


 \begin{abstract}
In this note,  we prove the uniqueness of blowups at  singular points of the no-sign obstacle problem $\Delta u=\chi_{_{B_1\backslash \{u=|Du|=0\}}}\ \text{in}\ B_1,$ thus give a positive answer to a problem raised in \cite[Notes of Chaper 7, page 149]{PSU12}. 
  \end{abstract}

 \maketitle

\baselineskip=15.8pt
\parskip=3pt

\section{Introduction}
In this note we consider the no-sign obstacle problem 
\begin{equation}\label{PA}
\Delta u=\chi_{\Omega}\ \text{in}\ B_1,
\end{equation}
where $\Omega=B_1\backslash\{u=|Du|=0\}.$ Denote by $\Gamma:=\partial\Omega\cap B_1$ the free boundary. Denote by $\Sigma$ the singular set of $\Gamma,$ namely,  $x^0\in\Sigma$ iff there exists a sequence $r_j\rightarrow 0$ such that
$u_{x^0,r_j}(x):=\frac{u(x^0+r_jx)}{r_j^2}$ converges to a homogeneous quadratic polynomial $q$ such that $\Delta q=1.$ Such $q$ is called a blowup of $u$ at $x^0.$
As pointed out in \cite[Notes of Chaper 7, page 149]{PSU12}, it is not known whether or not the blowup is unique at a singular point $x^0,$ namely, maybe $u_{x^0,r}$ will sub-converge to a different polynomial for another sequence $\tilde{r}_j\rightarrow 0.$ In this short note we will give a positive answer to this problem.
Denote $$\mathcal{Q}:=\{q(x)\ \text{homogeneous quadratic polynomial:}\ \Delta q=1\},$$
$$\mathcal{Q}^+:=\{q\in\mathcal{Q}: q\geq 0\}.$$
The main result we get is the following theorem.
\begin{theorem}\label{main}
Suppose $u$ solves \eqref{PA}, and $0\in \Sigma.$ Then $u_r(x):=\frac{u(rx)}{r^2}$ converges to a homogeneous quadratic polynomial $q_0\in \mathcal{Q}$ in $C^{1,\alpha}_{loc}(\mathbb{R}^n)$ for any $\alpha\in (0,1).$
\end{theorem}
Then, we can establish a uniform estimate of the convergence rate similarly to the classical obstacle problem.

The proofs of Theorem \ref{main} and Theorem \ref{UC} is based on Monneau's monotonicity formula, which works for the classical obstacle problem well. However, it was not expected that it  also works for the no-sign obstacle problem.
\begin{lemma}[Monneau's monotonicity formula] \label{mf}Suppose $u$ satisfies \eqref{PA}, and $0\in \Sigma.$ Then for any $q\in \mathcal{Q}^+,$ the functional
$$r\rightarrow M(r, u, q)=\frac{1}{r^{n+3}}\int_{\partial B_r}(u-q)^2dH^{n-1}$$
is monotone nondecreasing for $r\in (0, 1).$
\end{lemma}
Note that $M(r, u, q)=M(1, u_r, q).$ In the following we will call a constant universal if it depends only on $n$ and  $\|u\|_{L^\infty(B_1)}.$ 
\begin{remark}
It is well known that Monneau's monotonicity formula can be used to prove the uniqueness of blowup at singular points for the classical obstacle problem. However,  it is not known whether it works for the no-sign obstacle problem or not.
Indeed, it is mentioned in \cite[Notes of Chaper 7, page 149]{PSU12} that the only method known to work for the no-sign obstacle problem is the approach of \cite{Caf98}, adapted in \cite{CS04}
where the main tool used is the ACF monotonicity formula.
\end{remark}

\section{Proof of Theorem \ref{main}}
For the classical obstacle problem, suppose that there is a sequence $r_j\rightarrow 0$ such that $u_{r_j}\rightarrow q_0\in \mathcal{Q}^+,$ then by taking $q=q_0$ in Lemma \ref{mf} and using the monotonicity of $M(r, u, q_0),$ we have that $M(0+, u, q_0)=0.$
Then it easily follows that $u_r\rightarrow q_0$ in $C^{1,\alpha}_{loc}(\mathbb{R}^n)$ for any $\alpha\in (0,1).$ For the no-sign obstacle problem, a blowup $q_0$ may be non-positive, hence we can not take $q=q_0$ in Monneau's monotonicity formula. 

\begin{proof}[Proof of Theorem \ref{main}]
Suppose 
\begin{equation}\label{e1}
u_{r_j}\rightarrow q,\ u_{\tilde{r}_j}\rightarrow \tilde{q}\ \text{in}\ C^{1,\alpha}_{loc}(\mathbb{R}^n),
\end{equation}
for two sequences $r_j\rightarrow 0, \tilde{r}_j\rightarrow 0,$ where $q=\frac{1}{2}x\cdot Ax\in \mathcal{Q}$ and $\tilde{q}=\frac{1}{2}x\cdot \tilde{A}x\in \mathcal{Q}.$
Note that $A$ and $\tilde{A}$ are two symmetric matrices satisfying $Tr(A)=Tr(\tilde{A})=1.$
We only need to show that $A=\tilde{A}.$

Given any $\bar q=\frac{1}{2}x\cdot Bx\in \mathcal{Q}^+, $ by Lemma \ref{mf} we have that 
\begin{equation}\label{e2}
M(1, u_r, \bar q)=\int_{\partial{B_1}}(u_r-\bar q)^2dH^{n-1}\ \text{is monotone nondecreasing for}\ r\in (0, 1).
\end{equation}
By \eqref{e1} and \eqref{e2} we have that
\begin{equation}\label{matrix}
\int_{\partial B_1}\left(\frac{1}{2}x\cdot (A-B)x\right)^2dH^{n-1}=\int_{\partial B_1}\left(\frac{1}{2}x\cdot (\tilde A-B)x\right)^2dH^{n-1}.
\end{equation}
Since $A-\tilde{A}$ is symmetric,
by a rotation of coordinates we may assume $A-\tilde{A}$ is diagonalized with eigenvalues
$\lambda_1, \cdots, \lambda_n.$
Now, we choose $B=B^{t}=(b^t_{ij})_{i,j=1}^n,$ where 
$b^t_{11}=\frac{1}{2}-\frac{1}{2}t,$ $b^t_{22}=\frac{1}{2}+\frac{1}{2}t,$ and $b^{t}_{ij}=0$  for all other $i, j.$ 
It is easy to see that $B^{t}\geq 0$ and $Tr(B^{t})=1,$ provided $t\in (-1, 1),$ hence 
$\frac{1}{2}x\cdot B^{t}x\in \mathcal{Q}^+.$

Now, by \eqref{matrix} we have that 
\begin{equation}\label{e6}
f(t):=\int_{\partial B_1}\left(x\cdot (A-\tilde{A})x\right)\left(x\cdot (A+\tilde{A}-2B^{t})x\right)dH^{n-1}=0.
\end{equation}
 Hence 
\begin{equation}\label{e3}
x\cdot (A-\tilde{A})x=\sum_{i=1}^n\lambda_ix_i^2.
\end{equation}
Note that $Tr(A-\tilde{A})=\sum_{i=1}^n\lambda_i=0.$ 
It is also straightforward to compute that 
\begin{equation}\label{e4}
\frac{d\left(x\cdot (A+\tilde{A}-2B^{t})x\right)}{dt}=x_1^2-x_2^2.
\end{equation}

It follows from  \eqref{e6}, \eqref{e3} and \eqref{e4} that
\begin{eqnarray}
0=f'(t)=\frac{df}{dt}&=&\int_{\partial B_1}\left(\sum_{i=1}^n\lambda_ix_i^2\right)(x_1^2-x_2^2)\\
&=&\int_{\partial B_1}\left(\lambda_1x_1^4-\lambda_2x_2^4-(\lambda_1-\lambda_2)x_1^2x_2^2\right)dH^{n-1}\\
&+&\int_{\partial B_1}\left(\sum_{i=3}^n\lambda_ix_i^2\right)(x_1^2-x_2^2)dH^{n-1}
\end{eqnarray}
By the symmetry of $\partial B_1$ we have that $\int_{\partial B_1}x_1^4=\int_{\partial B_1}x_2^4,$ and that $\int_{\partial B_1}x_1^2x_i^2=\int_{\partial B_1}x_2^2x_i^2$ for all $i=3,\cdots, n.$
Hence from the above computation we have 
\begin{equation}\label{e8}
f'(t)=(\lambda_1-\lambda_2)\left(\int_{\partial B_1}x_1^4-\int_{\partial B_1}x_1^2x_2^2\right)=0.
\end{equation}
A direct computation shows that $\int_{\partial B_1}x_1^4=3\int_{\partial B_1}x_1^2x_2^2,$ hence
$\int_{\partial B_1}x_1^4-\int_{\partial B_1}x_1^2x_2^2\ne 0,$ which implies $\lambda_1=\lambda_2.$ 
Similarly, fix any $1\leq i_0<j_0\leq n,$
 we can choose $B^{t}=(b^t_{ij})_{i,j=1}^n,$ where 
$b^t_{i_0i_0}=\frac{1}{2}-\frac{1}{2}t,$ $b^t_{j_0j_0}=\frac{1}{2}+\frac{1}{2}t,$ and $b^{t}_{ij}=0$  for all other $i, j.$ The same argument gives that
$\lambda_{i_0}=\lambda_{j_0}.$ 
Hence $\lambda_1=\lambda_2=\cdots=\lambda_n.$ Using the fact that $\sum_{i=1}^n\lambda_i=0,$ we conclude that $\lambda_i=0$ for any $i=1,\cdots, n.$ Therefore $A=\tilde{A}.$

\end{proof}

\bibliographystyle{amsplain}

\end{document}